\documentclass{amsart}
\usepackage{amsmath}
\usepackage{amsthm}
\usepackage{amssymb}

\usepackage[T1]{fontenc}
\usepackage{lmodern}

\usepackage{diagrams}
\input xy
\xyoption{all}

 \newtheorem{thm}{Theorem}[section]
 \newtheorem{cor}[thm]{Corollary}
 
 \newtheorem{prop}[thm]{Proposition}
 \newtheorem{con}[thm]{Conjecture}
 \theoremstyle{definition}
 \newtheorem{defn}[thm]{Definition}
 \theoremstyle{remark}
 \newtheorem{rem}[thm]{Remark}
 \numberwithin{equation}{section}

\newcommand{\pn}{\noindent}

\newcommand{\ZZ}{\mathbb{Z}}
\newcommand{\QQ}{\mathbb{Q}}

\newcommand{\CC}{\mathbb{C}}
\newcommand{\GG}{\mathbb{G}}

\newcommand{\li}{\mathrm{Lie}\,}

\newcommand{\Hom}{\mathrm{Hom}}

\newcommand{\Ext}{\mathrm{Ext}}

\newcommand{\uAut}{\underline{\mathrm{Aut}}}

\newcommand{\bPicard}{\mathrm{\mathbf{Picard}}}

\newcommand{\coker}{\mathrm{coker}}
\newcommand{\im}{\mathrm{im}}

\newcommand{\W}{\mathrm{W}}
\newcommand{\F}{\mathrm{F}}
\newcommand{\T}{\mathrm{T}}
\newcommand{\h}{\mathrm{H}}
\newcommand{\R}{\mathrm{R}}

\newcommand{\spec}{{\mathrm{Spec}}\,}
\newcommand{\ok}{\overline k}

\newcommand{\gal}{{\mathrm{Gal}}(\ok/k)}
\newcommand{\dR}{\mathrm{dR}}

\newcommand{\cD}{\mathcal{D}}

\newcommand{\cK}{\mathcal{K}}

\newcommand{\bS}{\textbf{S}}
\newcommand{\pic}{\mathcal{P}}
\newcommand{\mhs}{\mathcal{MHS}}

\begin{document}

\title[extensions and 1-motives]
{Deligne's conjecture on extensions of 1-motives}

\author{Cristiana Bertolin}

\address{Dip. di Matematica, Universit\`a di Milano, Via C. Saldini 50, I-20133 Milano}
\email{cristiana.bertolin@googlemail.com}






\begin{abstract} 
We introduce the notion of extension of 1-motives.
 Using the dictionary between strictly commutative Picard stacks and complexes of abelian sheaves concentrated in degrees -1 and 0, we check that an extension of 1-motives induces an extension of the corresponding strictly commutative Picard stacks.
We compute the Hodge, the de Rham and the $\ell$-adic realizations of an extention of 1-motives. Using these results we can prove Deligne's conjecture on extensions of 1-motives.
\end{abstract}


\maketitle


\tableofcontents

\section*{Introduction}
Let $k$ be a field of characteristic 0 embeddable in $\CC.$  Let ${\mathcal{MR}}(k)$ be the Tannakian category of mixed realizations (for absolute Hodge cycles) over $k$.
In~\cite{D} Deligne defines the \emph{category of motives} as the Tannakian subcategory of 
${\mathcal{MR}}(k)$ generated by those mixed realizations coming from geometry.

A 1-motive $X=[L\stackrel{u}{\rightarrow}E]$ over $k$ is a geometrical object consisting of a finitely generated free $\ZZ\,$-module $L$, an extension $E$ of an abelian variety by a torus, and an homomorphism $u:L \rightarrow E$. To each 1-motive $X$ it is possible to associate its Hodge, its $\ell$-adic and its De Rham realization. These realizations together with the comparison isomorphisms build a mixed realization ${\T}(X)$ which is a motive because of the geometrical origin of $X$.

In~\cite{D} 2.4. Deligne writes: \emph{
Je conjecture que l'ensemble des motifs \`a coefficients entiers de la forme ${\T}(X)$, pour $X$ un 1-motif, est stable par extensions. Si ${\T}'$ est un motif \`a coefficients entiers, avec ${\T}' \otimes {\QQ} \stackrel{\sim}{\longrightarrow} {\T}(X)\otimes {\QQ},$ alors ${\T}'$ est de la forme ${\T}(X')$ avec $X'$ isog\`ene \`a $X$. La conjecture \'equivaut donc \`a ce que l'ensemble des motifs
${\T}(X)\otimes {\QQ}$, pour $X$ un 1-motif, soit stable par extension. Le mot ``conjecture'' est abusif en ce que l'\'enonc\'e n'a pas un sens pr\'ecis. Ce qui est conjectur\'e est que tout syst\`eme de r\'ealizations extension de ${\T}(X)$ par ${\T}(Y)$ ($X$ et $Y$ deux 1-motifs), et ``naturel'', ``provenant de la g\'eom\'etrie'', est isomorphe \`a celui d\'efini par un 1-motif $Z$ extension de $X$ par $Y$.} 

In order to explain this conjecture, Deligne furnishes the following example: 
Let $A$ be an abelian variety over $\QQ$. A point $a$ of $A({\QQ})$ defines a 1-motive $M=[{\ZZ} \stackrel{u}{\rightarrow} A]$ with $u(1)=a$.  The motive ${\T}(M)$, i.e. the mixed realization defined by $M$, is an extension of 
${\T}({\ZZ})$ by ${\T}(A).$ Therefore we have an arrow 
\begin{eqnarray*}
 \nonumber  A({\QQ}) & \longrightarrow  & {\Ext}^1({\T}({\ZZ}),{\T}(A))  \\
 \nonumber  a& \mapsto & {\T}(M)
\end{eqnarray*}
with the ${\Ext}^1$ computed in the abelian category of motives. Deligne's conjecture applied to ${\T}({\ZZ})$ and ${\T}(A)$ says that the above arrow is in fact a bijection:
\[ A({\QQ}) \cong {\Ext}^1({\T}({\ZZ}),{\T}(A)).
\]
In other words, any extension of ${\T}({\ZZ})$ by ${\T}(A)$ in the abelian category of motives (i.e. any mixed realization which is an extension of ${\T}({\ZZ})$ by ${\T}(A)$ and which comes from geometry) is defined by a unique point $a$ of $A({\QQ})$. The hypothesis \emph{"coming from geometry"} is essential (if we omit it, the conjecture is wrong: see remark \ref{essential}), but present technology gives no way to use it. Therefore, using \cite{BY83} (2.2.5), we reformulate Deligne's conjecture on extensions of 1-motives in the following way:

\begin{con}\label{cj}
Let $M_1$ and $M_2$ be two 1-motives defined over a field $k$ of characteristic 0 embeddable in $\CC$. There exists a bijection between 1-motives defined over $k$ modulo isogenies which are extensions of $M_1$ by $M_2$ and ${\Ext}^1_{\mathcal{M}(k)}({\T}(M_1) ,{\T}(M_2) )$ in the 
Tannakian subcategory $\mathcal{M}(k)$ of ${\mathcal{MR}}(k)$ generated by 1-motives:
\begin{eqnarray}
 \nonumber \varphi : \big\{ \mathrm{1-isomotive}~ M ~ \mathrm{extension~of}~M_1~ \mathrm{by}~M_2  \big\} & \stackrel{\cong}{\longrightarrow} & {\Ext}^1_{\mathcal{M}(k)}({\T}(M_1) ,{\T}(M_2) )  \\
 \nonumber M & \mapsto &  {\T}(M).
\end{eqnarray}
\end{con}

Recall that the Tannakian subcategory $\mathcal{M}(k)$ of ${\mathcal{MR}}(k)$ generated by 1-motives is the strictly full abelian subcategory of ${\mathcal{MR}}(k)$ which is generated by 1-motives by means of subquotients, direct sums, tensor products and duals. The aim of this paper is to prove the above conjecture.

This paper is organized as followed: 
in Section 1 we define the notion of extension of 1-motives. In section 2 we recall the notion of extension of strictly commutative Picard stacks, and using the dictionary between strictly commutative Picard stacks and complexes of abelian sheaves concentrated in degrees -1 and 0, we prove that an extension of 1-motives furnishes an extension of the corresponding strictly commutative Picard stacks. In Section 3 we show that there is a bijection between extensions of 1-motives and extensions of the corresponding Hodge realizations. For the $\ell$-adic and the De Rham realizations we don't have a bijection but just that extensions of 1-motives define extensions of the corresponding $\ell$-adic and De Rham realizations.
In Section 4 we prove Conjecture \ref{cj}.

The computation of the group of extensions of ${\T}({\ZZ})$ by ${\T}({\GG}_m)$ in the abelian category of motives
\[ {\GG}_m({\QQ}) \cong {\Ext}^1({\T}({\ZZ}),{\T}({\GG}_m))
\]
fits into the context of Beilinson's conjectures \cite{BL87} \S 5.

In \cite{BK07} Appendix C.9 Barbieri-Viale and Kahn provide a characterisation of the Yoneda Ext in the abelian category of 1-motives with torsion.

\section*{Acknowledgment}
The author is very grateful to Pierre Deligne who furnishes the good setting in which to work. I want to express my 
gratitude to Luca Barbieri-Viale for the various discussions we have had on 1-motives.

\section*{Notation}

Let $\bS$ be a site. 
Denote by $\cK(\bS)$ the category of complexes of abelian sheaves on the site $\bS$: all complexes that we consider in this paper are cochain complexes.
Let $\cK^{[-1,0]}(\bS)$ be the subcategory of $\cK(\bS)$ consisting of complexes $K=(K^i)_i$ such that $K^i=0$ for $i \not= -1$ or $0$. The good truncation $ \tau_{\leq n} K$ of a complex $K$ of $\cK(\bS)$ is the following complex: $ (\tau_{\leq n} K)^i= K^i$ for $i <n,  (\tau_{\leq n} K)^n= \ker(d^n)$
and $ (\tau_{\leq n} K)_i= 0$ for $i > n.$ For any $i \in {\ZZ}$, the shift functor $[i]:\cK(\bS) \rightarrow \cK(\bS) $ acts on a  complex $K=(K^n)_n$ as $(K[i])^n=K^{i+n}$ and $d^n_{K[i]}=(-1)^{i} d^{n+i}_{K}$.

Denote by $\cD(\bS)$ the derived category of the category of abelian sheaves on $\bS$, and let $\cD^{[-1,0]}(\bS)$ be the subcategory of $\cD(\bS)$ consisting of complexes $K$ such that ${\h}^i (K)=0$ for $i \not= -1$ or $0$. If $K$ and $K'$ are complexes of $\cD(\bS)$, the group ${\Ext}^i(K,K')$ is by definition ${\Hom}_{\cD(\bS)}(K,K'[i])$ for any $i \in {\ZZ}$. Let ${\R}{\Hom}(-,-)$ be the derived functor of the bifunctor ${\Hom}(-,-)$. The cohomology groups\\ ${\h}^i\big({\R}{\Hom}(K,K') \big)$ of 
${\R}{\Hom}(K,K')$ are isomorphic to ${\Hom}_{\cD(\bS)}(K,K'[i])$.

\section{Extensions of 1-motives}

Let $S$ be a scheme. 

A \emph{1-motive} $M=(X,A,T,G,u)$ over $S$ consists of
\begin{itemize}
    \item  an $S$-group scheme $X$ which is locally for the \'etale
topology a constant group scheme defined by a finitely generated free
$\ZZ \,$-module,
    \item an extension $G$ of an abelian $S$-scheme $A$ by an $S$-torus $T,$
    \item a morphism $u:X \rightarrow G$ of $S$-group schemes.
\end{itemize}

A 1-motive $M=(X,A,T,G,u)$ can be viewed also as a complex $[X \stackrel{u}{\rightarrow}G]$ of commutative $S$-group schemes with $X$ concentrated in degree -1 and $G$ concentrated in degree 0. A morphism of 1-motives is a morphism of complexes of commutative $S$-group schemes. Denote by $1-\mathrm{Mot}(S)$ the category of 1-motives over $S$. It is an additive category but \emph{it isn't an abelian category.}

Let $S={\spec}(k)$ be the spectrum of an algebraically closed field $k$.
Denote by $1-\mathrm{Isomot}(k)$ the $\QQ$-linear category associated to the category $1-\mathrm{Mot}(k)$ of 1-motives over $k$ (it has the same objects as $1-\mathrm{Mot}(k)$, but its sets of arrows are the sets of arrows of $1-\mathrm{Mot}(k)$ tensored with $\QQ$).
 The objects of $1-\mathrm{Isomot}(k)$ are called 1-isomotifs and the morphisms of $1-\mathrm{Mot}(k)$ which become isomorphisms in $1-\mathrm{Isomot}(k)$ are the
isogenies between 1-motives, i.e. the morphisms of complexes
$(f^{-1},f^0):[X  \rightarrow G] \rightarrow [X' \rightarrow G']$ such that
 $f^{-1}:X \rightarrow X'$ is injective with finite cokernel and
$f^{0}:G \rightarrow G'$ is surjective with finite kernel. \emph{The category $1-\mathrm{Isomot}(k)$ is an abelian category.} From now on, we write 1-motive instead of 1-isomotive, unless it is necessary to specify that we work modulo isogenies.

The results of this section are true for any base scheme $S$ such that 
the category $1-\mathrm{Isomot}(S)$ is abelian. 
Let $M_1=[X_1 \stackrel{u_1}{\rightarrow}G_1]$ and $M_2=[X_2 \stackrel{u_2}{\rightarrow}G_2]$ be two 1-motives defined over such a base scheme $S$.

\begin{defn}\label{def1-mot} An \emph{extension $(M,i,j)$ of $M_1$ by $M_2$} consists of a 1-motive $M=[X \stackrel{u}{\rightarrow}G]$ defined over $S$ and two morphisms of 1-motives 
$i=(i_{-1},i_0):M_2 \rightarrow M$ and $j=(j_{-1},j_0):M \rightarrow M_1$
\begin{equation}
\xymatrix{
 X_2 \; \ar@{^{(}->}[r]^{i_{-1}} \ar[d]_{u_2}& \; X \; \ar[d]^{u}\ar@{->>}[r]^{j_{-1}}& \;X_1 \ar[d]^{u_1}\\
 G_2 \; \ar@{^{(}->}[r]^{i_0}& \; G \; \ar@{->>}[r]^{j_0}& \; G_1 
}	
\end{equation}
such that 
\begin{itemize}
	\item $j_{-1} \circ i_{-1}=0, j_0 \circ i_0=0$,
	\item $i_{-1}$ and $i_0$ are injective,
	\item $j_{-1}$ and $j_0$ are surjective, and
	\item $u$ induces an isomorphism between the quotients $\ker(j_{-1})/ \im(i_{-1})$ and $\ker(j_0)/ \im(i_0)$.
\end{itemize}
\end{defn}

Often we will write only $M$ instead of $(M,i,j)$.

\section{Geometrical interpretation}

Let {\bS} be a site. A \emph{strictly commutative Picard {\bS}-stack} is an {\bS}-stack of groupoids $\pic$ endowed with a functor $ +: \pic \times_{\bS} \pic \rightarrow \pic, ~~(a,b) \mapsto a+b$, and two 
  natural isomorphisms of associativity $\sigma$ and of commutativity $\tau$,
such that for any object $U$ of $\bS$, $(\pic (U),+,\sigma, \tau)$ is a strictly commutative Picard category (see \cite{SGA4} 1.4.2 for more details).
An \emph{additive functor} $(F,\sum):\pic_1 \rightarrow \pic_2 $
between strictly commutative Picard {\bS}-stacks is a morphism of {\bS}-stacks endowed with a natural isomorphism $\sum: F(a+b) \cong F(a)+F(b)$ (for all $a,b \in \pic_1$) which is compatible with the natural isomorphisms $\sigma$ and $\tau$ of
$\pic_1$ and $\pic_2$. A \emph{morphism of additive functors $u:(F,\sum) \rightarrow (F',\sum') $} is an {\bS}-morphism of {\bS}-functors (see \cite{G} Chapter I 1.1) which is compatible with the natural isomorphisms $\sum$ and $\sum'$.

To any strictly commutative Picard {\bS}-stack $\pic$ we associate two abelian sheaves: $\pi_0(\pic)$ the sheaffification of the pre-sheaf which associates to each object $U$ of $\bS$ the group of isomorphism classes of objects of $\pic(U)$, and $\pi_1(\pic)$ the sheaf of automorphisms ${\uAut}(e)$ of the neutral object $e$ of $\pic$.

Denote by ${\bPicard}(\bS)$ the category whose objects are small strictly commutative Picard $\bS$-stacks and whose arrows are isomorphism classes of additive functors. In~\cite{SGA4} \S 1.4 Deligne constructs an equivalence of category  
\begin{equation}\label{st}
 st: \cD^{[-1,0]}(\bS) \longrightarrow  {\bPicard}(\bS) 
\end{equation}
which furnishes a dictionary between strictly commutative Picard $\bS$-stacks and complexes of abelian sheaves on $\bS$. We denote by $[\,\,]$ the inverse equivalence of $st$. 

An \emph{extension $\pic=(\pic,I:\pic_2 \rightarrow \pic, J:\pic \rightarrow \pic_1) $ of $\pic_1$ by $\pic_2$} consists of a strictly commutative Picard $\bS$-stack $\pic$, two additive functors $I:\pic_2 \rightarrow \pic$ and $ J:\pic \rightarrow \pic_1$, and an isomorphism of additive functors between the composite $J \circ I$ and the trivial additive functor: $J \circ I \cong 0,$ such that the following equivalent conditions are satisfied:
   \begin{description}
     \item[(a)] $\pi_0(J): \pi_0(\pic) \rightarrow \pi_0(\pic_1)$ is surjective and $I$ induces an equivalence of strictly commutative Picard $\bS$-stacks between $\pic_2$ and $\ker(J)$;
     \item[(b)] $\pi_1(I): \pi_1(\pic_2) \rightarrow \pi_1(\pic)$ is injective and $J$ induces an equivalence of strictly commutative Picard $\bS$-stacks between $\coker(I)$ and $\pic_1$. 
   \end{description}
(see \cite{Be10} for the definition of kernel and cokernel of an additive functor).
Let $K=[K^{-1} \stackrel{d^K}{\rightarrow}K^0]$ and $L=[L^{-1} \stackrel{d^L}{\rightarrow}L^0]$ be complexes of $\cK^{[-1,0]}(\bS)$, and let $F: st(K)  \rightarrow st(L)$ be an additive functor induced by a morphism of complexes $f: K \rightarrow L$ in $\cK^{[-1,0]}(\bS)$. According to \cite{Be10} Lemma 3.4 we have 
\begin{eqnarray}
\label{eq:[ker]} [\ker(F)] & = & \tau_{ \leq 0}\big( MC(f)[-1]\big) =\big[K^{-1} ~ \stackrel{(f^{-1},-d^K)}{\longrightarrow} ~  \ker(d^L,f^0)\big]\\
\label{eq:[coker]} [\coker(F)] & = & \tau_{\geq -1}MC(f) =\big[\coker(f^{-1},-d^K) ~ \stackrel{(d^L,f^0)}{\longrightarrow} ~ L^0 \big]
\end{eqnarray}
 where $MC(f)$ is the mapping cone of the morphism $f$. Hence we have the following

\begin{cor}\label{cor:extinD}
Let
\[ K \stackrel{i}{\longrightarrow} L \stackrel{j}{\longrightarrow} M \]
be morphisms of complexes of $\cK^{[-1,0]}(\bS)$ and denote by
$I$ and $J$
the additive functors induced by $i$ and $j$ respectively.
Then the strictly commutative Picard $\bS$-stack $st(L)=(st(L),I,J)$ is an extension of $st(M)$ by $st(K)$ if and only if $j \circ i=0$ and 
the following equivalent conditions are satisfied: 
\begin{description}
	\item[(a)] ${\h}^{0}(j): {\h}^{0}(L) \rightarrow {\h}^{0}(M)$ is surjective and $i$ induces a quasi-iso\-mor\-phism between $K$ and $ \tau_{\leq 0} (MC(j)[-1])$;
      \item[(b)] ${\h}^{-1}(i): {\h}^{-1}(K) \rightarrow {\h}^{-1}(L)$ is injective and $j$ induces a quasi-iso\-mor\-phism between $ \tau_{\geq -1} MC(i)$ and $M$.
\end{description}
\end{cor}

Let $S$ be a scheme. From now on the site $\bS$ is the big fppf site over $S$.

\begin{prop}\label{prop:1-motpicstack} Let $M_1=[X_1 \stackrel{u_1}{\rightarrow}G_1]$ and $M_2=[X_2 \stackrel{u_2}{\rightarrow}G_2]$ be two 1-motives defined over $S$. If $M=[X \stackrel{u}{\rightarrow}G]$ is an extension of $M_1$ by $M_2$, then $st(M)$ is an extension of $st(M_1)$ by $st(M_2).$ 
\end{prop}

\begin{proof} Denote by
$(i_{-1},i_0):M_2 \rightarrow M$ and $(j_{-1},j_0):M \rightarrow M_1$ the morphisms of 1-motives underlying the extension $M$ of $M_1$ by $M_2$.
These morphisms furnish two additive functors:
$$ I:st(M_2) \longrightarrow st(M) \qquad \mathrm{and} \qquad  J:st(M) \longrightarrow st(M_1).$$
First observe that the conditions $j_{-1} \circ i_{-1}=0$ and $ j_0 \circ i_0=0$  imply that $J \circ I \cong 0.$ Remark also that since $j_0:G \rightarrow G_1$ is surjective, also the morphism ${\h}^0(j):G/u(X) \rightarrow G_1/u_1(X_1)$ is surjective. \\
By Corollary \ref{cor:extinD}, it remains to prove that the morphism of complexes $i$ induces a quasi-isomorphism between $M_2$ and $\tau_{\leq 0} (MC(j)[-1])$.
 Explicitly $\tau_{\leq 0} (MC(j)[-1])$ is the complex 
\[[ X \stackrel{k}{\longrightarrow}\ker(u_1,j_0) ] \]
 with $k: X \rightarrow  \ker(u_1,j_0)$ the morphism induced by $(j_{-1},-u): X \rightarrow  X_1+G$, and so we have to prove that $(i_{-1},i_0)$ induces the quasi-isomorphisms
 
\begin{align}
\label{eq:1}\ker(u_2) &\cong \ker(k), \\
\label{eq:2} G_2 /u_2(X_2)& \cong \ker(u_1,j_0) / (j_{-1},-u)(X) .
\end{align}
We start with the first isomorphism. Because of the commutativity of the first square of diagram (\ref{def1-mot}), $i_{-1}(\ker(u_2))$ is contained in $\ker(u)$. Since $ j_{-1} \circ i_{-1}=0,$   $i_{-1}(\ker(u_2))$ is contained also in $\ker(j_{-1})$ and so 
we have the inclusion $i_{-1}(\ker(u_2)) \subseteq \ker(k).$
The isomorphism between the quotients $\ker(j_{-1})/ \im(i_{-1})$ and $\ker(j_0)/ \im(i_0)$ induces the exact sequence
\[ 0 \longrightarrow X_2 \stackrel{i_{-1}}{\longrightarrow} \ker(j_{-1}) \stackrel{u}{\longrightarrow} \ker(j_0)/ \im(i_0) \longrightarrow 0 \]
Therefore we have the equality $ \ker(k) \subseteq i_{-1}(X_2)$.
Now because of the commutativity of the first square of diagram (\ref{def1-mot}) and because of the injectivity of $i_0$ we have that $i_{-1}( \ker(u_2))$ contains  
$ \ker(k)$. Hence we can conclude that via the morphism $i_{-1},$ $\ker(u_2)$ and $\ker(k)$ are isomorphic.\\
Concerning the second isomorphism (\ref{eq:2}), since $j_{-1}: X \rightarrow X_1$ is surjective, we have the isomorphism  $\ker(u_1,j_0) / (j_{-1},-u)(X) \cong \ker(j_0) / u(X)$, and so we have to prove that the morphism $ i_0: G_2 \rightarrow  G$ induces an isomorphism
$$G_2 /u_2(X_2) \cong \ker(j_0) / u(X).$$
Since the morphism $u: X \rightarrow G$ induces the isomorphism  $\ker(j_{-1})/ i_{-1}(X_2) \cong \ker(j_0)/ i_0(G_2)$, the composite of the injection 
$ i_0: G_2 \rightarrow  \ker(j_0)$ with the projection $ \ker(j_0) \rightarrow  \ker(j_0) / u(X)$ furnishes the surjection
$$ p: G_2 \longrightarrow \ker(j_0) / u(X). $$
Because of the commutativity of the first square of diagram (\ref{def1-mot}),
 $u_2(X_2)$ is contained in $\ker(p).$
On the other hand $i_0 (\ker (p))$ is contained in $u(X)$.
The isomorphism $\ker(j_{-1})/ i_{-1}(X_2) \cong \ker(j_0)/ i_0(G_2)$ implies that in fact $i_0 (\ker (p))$ is contained in $u(i_{-1}(X_2)).$
Because of the commutativity of the first square of diagram (\ref{def1-mot}) and because of the injectivity of $i_0,$ $ \ker(p)$ is contained in $u_2(X_2)$. Hence via the morphism $i_{0},$ $G_2 /u_2(X_2)$ and $\ker(j_0) / u(X)$ are isomorphic.
\end{proof}

By the above proposition, the group law for extensions of strictly commutative Picard $\bS$-stacks defined in \cite{Be10} \S 4 furnishes a group law for extensions of 1-motives. The neutral object with respect to this group law on the set of isomorphism classes of extensions of $M_1=[X_1 \stackrel{u_1}{\rightarrow}G_1]$ by $M_2=[X_2 \stackrel{u_2}{\rightarrow}G_2]$ is the 1-motive $M_1 + M_2=[X_1 \times X_2 \stackrel{(u_1,u_2)}{\rightarrow}G_1\times G_2].$

\section{Transcendental and algebraic interpretations}

First we recall briefly the construction of the Hodge, De Rham and $\ell$-adic realizations of a 1-motive $M=(X,A,T,G,u)$ defined over $S$ (see~\cite{D1} \S 10.1 for more details):
\begin{itemize}
	\item if $S$ is the spectrum of the field $\CC$ of complex numbers, the \emph{Hodge realization} ${\T}_{\h}(M)=({\T}_{\ZZ}(M), {\W}_*,{\F}^*)$ of $M$ is the mixed Hodge structure consisting of the fibred product ${\T}_{\ZZ}(M)={\rm Lie}(G)\times_{G} X $ (viewing ${\rm Lie}(G)$ over $G$ via the exponential map and $X$ over $G$ via $u$) and of the weight and Hodge filtrations defined in the following way:
\begin{eqnarray}
\nonumber  {\W}_{0}({\T}_{\ZZ}(M)) &=& {\T}_{\ZZ}(M), \\
\nonumber  {\W}_{-1}({\T}_{\ZZ}(M)) &=& {\h}_1(G,\ZZ), \\
\nonumber  {\W}_{-2}({\T}_{\ZZ}(M)) &=& {\h}_1( T,\ZZ), \\
\nonumber  {\F}^0 ({\T}_{\ZZ}(M) \otimes {\CC}) &=& \ker\big( {\T}_{\ZZ}(M) \otimes {\CC} \longrightarrow {\rm Lie}(G)\big).
\end{eqnarray}
\item if $S$ is the spectrum of a field $k$ of characteristic 0 embeddable in $\CC$,
the \emph{$\ell$-adic realization} ${\T}_{\ell}(M)$ of the 1-motive $M$ is the projective limit of the ${\ZZ}/\ell^n{\ZZ}\,$-modules 
\[ {\T}_{{\ZZ}/\ell^n{\ZZ}}(M) =  \big\{(x,g) \in X \times G ~~\vert~~ u(x)=\ell^n \, g\big\} \big/ \big\{(\ell^n \, x ,u(x))~~\vert~~ x\in X\big\}.\] 
\item if $S$ is the spectrum of a field $k$
of characteristic 0 embeddable in $\CC,$ the \emph{de Rham realization} ${\T}_{\dR}(M)$ of $M$ is the Lie algebra of $G^\natural$ where $M^\natural=[X \rightarrow G^\natural]$ is the universal vectorial extension of $M$ by the vectorial group ${\Ext}^1(M,{\GG}_a)^*$. The Hodge filtration on ${\T}_{\dR}(M)$ is defined by ${\F}^0{\T}_{\dR}(M)= \ker ( {\li}G^\natural \rightarrow {\li}G ).$
\end{itemize}

\begin{prop}\label{prop:hodge}
Let $M_1=[X_1 \stackrel{u_1}{\rightarrow}G_1]$ and $M_2=[X_2 \stackrel{u_2}{\rightarrow}G_2]$ be two 1-motives defined over $\CC$. 
\begin{enumerate}
	\item If $M=[X \stackrel{u}{\rightarrow}G]$ is an extension of $M_1$ by $M_2$, then ${\T}_{\h}(M)$ is an extension of ${\T}_{\h}(M_1)$ by ${\T}_{\h}(M_2)$ in the abelian category $\mhs$ of mixed Hodge structures.
      \item Let $E$ be an extension of ${\T}_{\h}(M_1)$ by ${\T}_{\h}(M_2)$ in the category $\mhs$. Then modulo isogenies, there exists a unique extension $M$ of $M_1$ by $M_2$ which defines the isomorphism class of the extension $E$ i.e. such that ${\T}_{\h}(M)$ and $E$ are isomorphic in $\mhs$ as extensions of ${\T}_{\h}(M_1)$ by ${\T}_{\h}(M_2).$ 
\end{enumerate}
In other words, we have a bijection
\begin{eqnarray}
 \nonumber \varphi : \big\{ \mathrm{1-isomotive}~ M ~ \mathrm{extension~of}~M_1~ \mathrm{by}~M_2  \big\} & \stackrel{\cong}{\longrightarrow} & {\Ext}^1_{\mhs}({\T}_{\h}(M_1) ,{\T}_{\h}(M_2) )  \\
 \nonumber M & \mapsto &  {\T}_{\h}(M).
\end{eqnarray}
\end{prop}

\begin{proof} (1) Denote by $i=(i_{-1},i_0): M_2 \rightarrow M$ and $j=(j_{-1},j_0): M \rightarrow M_1$ the morphisms of 1-motives underlying the extension $M=(M,i,j)$.
By Proposition \ref{prop:1-motpicstack}, the strictly commutative Picard \bS-stack $st(M)$ is an extension of $st(M_1)$ by $st(M_2)$.
Corollary \ref{cor:extinD} implies that via $i$ the complexes $M_2$ and $ \tau_{\leq 0} (MC(j)[-1])$ are isomorphic in the derived category ${\cD}(\bS)$, and so,  
via the morphism ${\T}_{\h}(i_{-1},i_0)$ induced by $i=(i_{-1},i_0)$,
their Hodge realizations are isomorphic in the category $\mhs$:
\[ {\T}_{\h}(i_{-1},i_0): {\T}_{\h}(M_2) \stackrel{\cong}{\longrightarrow} {\T}_{\h}(\tau_{\leq 0} (MC(j)[-1])). \]
Explicitly the $\ZZ \,$-module underlying the Hodge realization of 
$\tau_{\leq 0} (MC(j)[-1])$ is
\begin{eqnarray} \label{conehodge}
 {\T}_{\ZZ}(\tau_{\leq 0} (MC(j)[-1])) &=& \li(\ker(u_1,j_0))\times_{\ker(u_1,j_0)} X \\
\nonumber  &=& \big(\li(\ker(j_0))  \oplus \ker(u_1)\big)\times_{\ker(u_1,j_0)} X
\end{eqnarray}
The morphism of 1-motive $j=(j_{-1},j_0):M \rightarrow M_1 $ induces 
a morphism $ {\T}_{\h}(j_{-1},j_0): {\T}_{\h}(M) \rightarrow {\T}_{\h}(M_1) $
 between the Hodge realizations of $M$ and $M_1$. To have this morphism is the same as to have the morphisms $\li({j_0}): \li(G) \rightarrow \li(G_1)$ and $j_{-1}: X \rightarrow X_1$ such that the following diagram commute
\begin{equation}
\xymatrix{
& \li(G)\ar[r]^{\li(j_0)}& \li(G_1) \ar[dr]^{exp}& \\
\li(G)\times_{G} X  \quad \ar[ur]^{pr} \ar[r]^{{\T}_{\ZZ}(j_{-1},j_0)} \ar[dr]_{pr} &\quad \li(G_1)\times_{G_1} X_1 \ar[ur]^{pr} \ar[dr]_{pr} & &G_1\\
 &X\ar[r]_{j_{-1}}& X_1 \ar[ur]_{u_1}& 
}
\label{eq:diagramhodge}
\end{equation}
where $pr$ are the projections and $exp$ the exponential map.
Since the morphisms $j_{-1}:X \rightarrow X_1$ and $j_0:G \rightarrow G_1$ are surjective, also the morphism ${\T}_{\h}(j_{-1},j_0)$ is surjective. Moreover the equality (\ref{conehodge}) implies that the mixed Hodge structure ${\T}_{\h}(\tau_{\leq 0} (MC(j)[-1]))$ is the kernel of ${\T}_{\h}(j_{-1},j_0): {\T}_{\h}(M) \rightarrow {\T}_{\h}(M_1).$ Hence we have an exact sequence in the category $\mhs$
\[0 \longrightarrow {\T}_{\h}(M_2) \stackrel{{\T}_{\h}(i_{-1},i_0)}\longrightarrow {\T}_{\h}(M) \stackrel{{\T}_{\h}(j_{-1},j_0)}\longrightarrow {\T}_{\h}(M_1) \longrightarrow 0 . \]
Setting $\varphi(M)= {\T}_{\h}(M)$ we have construct an arrow 
\[\varphi : \big\{ \mathrm{1-isomotive}~ M ~ \mathrm{extension~of}~M_1~ \mathrm{by}~M_2  \big\}  \longrightarrow {\Ext}^1_{\mhs}({\T}_{\h}(M_1) ,{\T}_{\h}(M_2) )\]
 which is well defined: isogeneous 1-motives which are extensions of $M_1$ by $M_2$ define the same isomorphism class of extensions of ${\T}_{\h}(M_1)$ by ${\T}_{\h}(M_2)$. The reader can check that the arrow $\varphi$ is in fact an homomorphism, i.e. it respects the group law of extensions of 1-motives and the group law of extensions of mixed Hodge structures.  \\
(2) Now we prove that $\varphi$ is a bijection.\\
Injectivity of $\varphi$ : Let $M$ be a 1-motive extension of $M_1$ by $M_2$ and suppose that 
$\varphi(M)$ is the zero object of ${\Ext}^1_{\mhs}({\T}_{\h}(M_1),{\T}_{\h}(M_2))$. We have
\begin{eqnarray}
\nonumber  {\T}_{\h}(M) &=& {\T}_{\h}(M_1) \oplus {\T}_{\h}(M_2), \\
\nonumber  &=& \li(G_1 \times G_2)\times_{G_1 \times G_2} ( X_1 \times X_2).
\end{eqnarray}
Therefore the 1-motives $M$ and 
$[X_1 \times  X_2  \stackrel{u_1 \times u_2}{\rightarrow}G_1 \times G_2]$ have the same Hodge realization and so they are isogeneous. \\
Surjectivity of $\varphi$ : Now suppose to have an extension $E$ of ${\T}_{\h}(M_1)$ by ${\T}_{\h}(M_2)$ in the category $\mhs$
\[0 \longrightarrow {\T}_{\h}(M_2) \stackrel{f}\longrightarrow E \stackrel{g}\longrightarrow {\T}_{\h}(M_1) \longrightarrow 0 . \]
Since ${\T}_{\h}(M_1)$ and ${\T}_{\h}(M_2)$ are mixed Hodge structures of type $\{0,0\}, \{-1,0\},\\
 \{0,-1\},\{-1,-1\}$ also $E$ must be of this type. Therefore according to the equivalence of category~\cite{D1} (10.1.3), there exists a 1-motive $M$ and morphisms of 1-motives $i=(i_{-1},i_0): M_2 \rightarrow M, j=(j_{-1},j_0): M \rightarrow M_1$  such that ${\T}_{\h}(M)=E$ and ${\T}_{\h}(i)=f, {\T}_{\h}(j)=g.$ It remains to check that $(M,i,j)$ is an extension of $M_1$ by $M_2$. Since $g \circ f=0$, it is clear that $j \circ i=0$.  
Because of the commutative diagram (\ref{eq:diagramhodge}), the surjectivity of $g$ implies the surjectivity of $j_0:G \rightarrow G_1$ and of $j_{-1}:X \rightarrow X_1$. Doing an analogous commutative diagram for the morphism $f={\T}_{\h}(i) :\li(G_2)\times_{G_2} X_2 \rightarrow \li(G)\times_{G} X $, we see that the injectivity of $f$ implies the injectivity of $i_0:G_2 \rightarrow G$ and of $i_{-1}:X_2 \rightarrow X$. 
Let now $m$ be an element of ${\T}_{\h}(M) = \li(G)\times_{G} X $. We have that ${\T}_{\h}(j)(m)=0$ if the projection $pr_{\li(G)}(m)$ of $m$ on $\li(G)$ lies in $\ker(\li(j_0))$, and the projection $pr_X(m)$ of $m$ on $X$ lies in $ \ker(j_{-1})$. Hence
 the morphism $u:X \rightarrow G$ has to induce an isomorphism between $\ker(j_{-1})/ \im(i_{-1})$ and $\ker(j_0)/ \im(i_0)$.
\end{proof}

\begin{prop}\label{prop:l-adic}
Let $M_1=[X_1 \stackrel{u_1}{\rightarrow}G_1]$ and $M_2=[X_2 \stackrel{u_2}{\rightarrow}G_2]$ be two 1-motives defined over a field $k$ of characteristic 0 embeddable in $\CC$. 
 If $M=[X \stackrel{u}{\rightarrow}G]$ is an extension of $M_1$ by $M_2$, then ${\T}_{\ell}(M)$ is an extension of ${\T}_{\ell}(M_1)$ by ${\T}_{\ell}(M_2)$.
\end{prop}

\begin{proof} Denote by $i=(i_{-1},i_0): M_2 \rightarrow M$ and $j=(j_{-1},j_0): M \rightarrow M_1$ the morphisms of 1-motives underlying the extension $M=(M,i,j)$.
By Proposition \ref{prop:1-motpicstack}, the strictly commutative Picard \bS-stack $st(M)$ is an extension of $st(M_1)$ by $st(M_2)$. Corollary \ref{cor:extinD} implies that via $i$ the complexes $M_2$ and $ \tau_{\leq 0} (MC(j)[-1])$ are isomorphic in the derived category ${\cD}(\bS)$, and so,
via the morphism ${\T}_{\ell}(i_{-1},i_0)$ induced by $i=(i_{-1},i_0)$,
their $\ell$-adic realizations are isomorphic:
\[ {\T}_{\ell}(i_{-1},i_0): {\T}_{\ell}(M_2) \stackrel{\cong}{\longrightarrow} {\T}_{\ell}(\tau_{\leq 0} (MC(j)[-1])). \]
Explicitly the $\ell$-adic realization of $\tau_{\leq 0} (MC(j)[-1])$ is
 the projective limit of the ${\ZZ}/\ell^n{\ZZ}\,$-modules
\begin{equation} \label{conel-adic}
 {\T}_{{\ZZ}/\ell^n{\ZZ}}(\tau_{\leq 0} (MC(j)[-1])) = 
 \end{equation}
 \[ \big\{(x,(z,g)) \in X \times \ker(u_1,j_0) ~~\vert~~ (j_{-1},-u)(x)=\ell^n \, (z,g)\big\} \big/ \big\{(\ell^n \, x ,(j_{-1},-u)(x))~~\vert~~ x\in X\big\} 
\]
The morphism of 1-motive $j=(j_{-1},j_0):M \rightarrow M_1 $ induces 
a morphism $ {\T}_{\ell}(j_{-1},j_0): {\T}_{\ell}(M) \rightarrow {\T}_{\ell}(M_1)$
between the $\ell$-adic realizations of $M$ and $M_1$.
Since the morphisms $j_{-1}:X \rightarrow X_1$ and $j_0:G \rightarrow G_1$ are surjective, also the morphism ${\T}_{\ell}(j_{-1},j_0)$ is surjective. Moreover from the equality (\ref{conel-adic}) we get that the ${\QQ}_\ell \, $-vector space ${\T}_{\ell}(\tau_{\leq 0} (MC(j)[-1]))$ is the kernel of the morphism ${\T}_{\ell}(j_{-1},j_0): {\T}_{\ell}(M) \rightarrow {\T}_{\ell}(M_1).$ Hence we have an exact sequence 
\[0 \longrightarrow {\T}_{\ell}(M_2) \stackrel{{\T}_{\ell}(i_{-1},i_0)}\longrightarrow {\T}_{\ell}(M) \stackrel{{\T}_{\ell}(j_{-1},j_0)}\longrightarrow {\T}_{\ell}(M_1) \longrightarrow 0 . \]
\end{proof}

\begin{prop}\label{prop:derham}
Let $M_1=[X_1 \stackrel{u_1}{\rightarrow}G_1]$ and $M_2=[X_2 \stackrel{u_2}{\rightarrow}G_2]$ be two 1-motives defined over a field $k$ of characteristic 0 embeddable in $\CC$. 
 If $M=[X \stackrel{u}{\rightarrow}G]$ is an extension of $M_1$ by $M_2$, then ${\T}_{\dR}(M)$ is an extension of ${\T}_{\dR}(M_1)$ by ${\T}_{\dR}(M_2)$.
\end{prop}

\begin{proof} Denote by $i=(i_{-1},i_0): M_2 \rightarrow M$ and $j=(j_{-1},j_0): M \rightarrow M_1$ the morphisms of 1-motives underlying the extension $M=(M,i,j)$.
By Proposition \ref{prop:1-motpicstack}, the strictly commutative Picard \bS-stack $st(M)$ is an extension of $st(M_1)$ by $st(M_2)$. Corollary \ref{cor:extinD} implies that via $i$ the complexes $M_2$ and $ \tau_{\leq 0} (MC(j)[-1])$ are isomorphic in the derived category ${\cD}(\bS)$, and so, 
via the morphism ${\T}_{\dR}(i_{-1},i_0)$ induced by $i=(i_{-1},i_0)$,
their de Rham realizations are isomorphic:
\[ {\T}_{\dR}(i_{-1},i_0): {\T}_{\dR}(M_2) \stackrel{\cong}{\longrightarrow} {\T}_{\dR}(\tau_{\leq 0} (MC(j)[-1])). \]
Explicitly the de Rham realization of the 1-motive $\tau_{\leq 0} (MC(j)[-1])$ is
\begin{eqnarray} \label{conederham}
 {\T}_{\dR}(\tau_{\leq 0} (MC(j)[-1])) &=& \li\big( \ker(u_1,j_0)^\natural \big)\\
 \nonumber  &=& \li\big( \ker(j_0)^\natural \big)  \oplus \big( \ker(u_1) \otimes k\big)
\end{eqnarray}
where $(\tau_{\leq 0} (MC(j)[-1]))^\natural=[X \rightarrow \ker(u_1,j_0)^\natural]$ is the universal vectorial extension of $\tau_{\leq 0} (MC(j)[-1])$ by the vectorial group ${\Ext}^1(\tau_{\leq 0} (MC(j)[-1]),{\GG}_{a})^*.$
The morphism of 1-motive $j=(j_{-1},j_0):M \rightarrow M_1 $ induces 
a morphism ${\T}_{\dR}(j_{-1},j_0): {\T}_{\dR}(M) \rightarrow {\T}_{\dR}(M_1)$
between the de Rham realizations of $M$ and $M_1$. Explicitly we have the following commutative diagram
\[
\xymatrix{
 & {\T}_{\dR}(M)= \li (G^\natural)  \ar[dr]^{exp} &  & & X \ar[dll] \ar@{=}[r] & X \ar[dll]^{u} \ar[ddd]^{j_{-1}}\\
0 \ar[r] & {\Ext}^1(\tau_{\leq 0} (M,{\GG}_{a}))^* \quad  \ar[r] \ar[d] &\quad G^\natural \ar[r] \ar[d] &  G \ar[r] \ar[d]^{j_0} & 0 &\\
0 \ar[r] & {\Ext}^1(\tau_{\leq 0} (M_1,{\GG}_{a}))^*  \quad  \ar[r]  &\quad G_1^\natural \ar[r]  &  G_1 \ar[r] & 0  &\\
& {\T}_{\dR}(M_1) =\li (G^\natural_1)  \ar[ur]_{exp} &  & & X_1 \ar[ull] \ar@{=}[r] & X_1 \ar[ull]_{u_1}\\
}
\]
where $exp$ are the exponential maps.
Since the morphisms $j_{-1}:X \rightarrow X_1$ and $j_0:G \rightarrow G_1$ are surjective, also the morphism ${\T}_{\dR}(j_{-1},j_0)$ is surjective. Moreover the equality (\ref{conederham}) implies    that the $k$-vector space  ${\T}_{\dR}(\tau_{\leq 0} (MC(j)[-1]))$ is the kernel of ${\T}_{\dR}(j_{-1},j_0): {\T}_{\dR}(M) \rightarrow {\T}_{\dR}(M_1).$ Hence we have an exact sequence 
\[0 \longrightarrow {\T}_{\dR}(M_2) \stackrel{{\T}_{\dR}(i_{-1},i_0)}\longrightarrow {\T}_{\dR}(M) \stackrel{{\T}_{\dR}(j_{-1},j_0)}\longrightarrow {\T}_{\dR}(M_1) \longrightarrow 0 . \]
\end{proof}

\section{Proof of the conjecture}

Let $S$ be the spectrum of a field $k$ of characteristic 0 embeddable in $\CC.$ Fix an algebraic closure $\ok$ of $k$. Let ${\mathcal{MR}}(k)$ be the neutral Tannakian category over $\QQ$ of mixed realizations (for absolute Hodge cycles) over $k$.
The objects of ${\mathcal{MR}}(k)$ are families
\[N=((N_\sigma, \mathcal{L}_\sigma),N_{\dR},N_{\ell}, I_{\sigma, {\dR}}, I_{{\overline \sigma}, {\ell}} )_{\ell,\sigma,{\overline \sigma}}\]
where
\begin{itemize}
  \item $N_\sigma$ is a mixed Hodge structure for any embedding $\sigma:k \rightarrow \CC$ of $k$ in $\CC$;
  \item $N_{\dR}$ is a finite dimensional $k$-vector space with an increasing filtration ${\W}_*$ (the Weight filtration) and a decreasing filtration ${\F}^*$ (the Hodge filtration);
  \item $N_{\ell}$ is a finite-dimensional $\QQ_\ell$-vector space with a continuous $\gal$-action and an increasing filtration ${\W}_*$ (the Weight filtration), which is $\gal$-equivariant, for any prime number $\ell$;
  \item $I_{\sigma, {\dR}}:N_\sigma \otimes_{\QQ} {\CC} \rightarrow N_{\dR}\otimes_k {\CC} $
and $I_{{\overline \sigma}, {\ell}}:N_\sigma \otimes_{\QQ}
{\QQ}_\ell \rightarrow N_{\ell}$ are comparison isomorphisms for any $\ell$, any $\sigma$ and any
${\overline \sigma}$ extension of $\sigma$ to the algebraic closure of $k$;
  \item $\mathcal{L}_\sigma$ is a lattice in $ N_\sigma$ such that,
for any prime number $\ell$, the image $\mathcal{L}_\sigma \otimes {\ZZ}_\ell$
of this lattice through the comparison isomorphism $I_{{\overline \sigma}, {\ell}}$ is a $\gal$-invariant subgroup of $N_{\ell}$ ($\mathcal{L}_\sigma$ is the integral structure of the object $N$ of ${\mathcal{MR}}_{\ZZ}(k)$).
\end{itemize}

According to~\cite{D1} (10.1.3) we have the fully faithful functor
\begin{eqnarray}
\nonumber  1-\mathrm{Mot}(k) & \longrightarrow & {\mathcal{MR}}(k) \\
\nonumber  M & \longmapsto &
{\T}(M)=({\T}_\sigma(M),{\T}_{\rm dR}(M),{\T}_{\ell}(M),
I_{\sigma, {\rm dR}}, I_{{\overline \sigma}, {\ell}} )_{\ell, \sigma,
  {\overline \sigma}}
\end{eqnarray}
\pn which attaches to each 1-motive $M$
its Hodge realization ${\T}_\sigma(M)$ for any embedding
 $\sigma:k \rightarrow \CC$ of $k$ in $\CC$, its de Rham realization ${\rm T}_{\dR}(M)$, its $\mathbf{\ell}$-adic realization ${\T}_{\ell}(M)$ for any prime number $\ell$, and its
comparison isomorphisms. Denote by $\mathcal{M}(k)$ the Tannakian subcategory of ${\mathcal{MR}}(k)$ generated by 1-motives, i.e. the strictly full abelian subcategory of ${\mathcal{MR}}(k)$ which is generated by 1-motives by means of subquotients, direct sums, tensor products and duals.
Recall that according to \cite{BY83} (2.2.5), any embedding $\sigma:k \rightarrow \CC$ of $k$ in $\CC$ furnishes a fully faithful functor from $\mathcal{M}(k)$ to the category $\mhs$ of mixed Hodge structures. 

We can now prove Conjecture \ref{cj}:

 \begin{proof}
Denote by ${\T}(M_i)=({\T}_\sigma(M_1),{\T}_{\rm dR}(M_1),{\T}_{\ell}(M_1),
I_{\sigma, {\rm dR}}, I_{{\overline \sigma}, {\ell}} )$ (for $i=1,2$) the system of realization defined by $M_i$ for $i=1,2$. Consider an extension of ${\T}(M_1)$ by ${\T}(M_2)$ in the category $\mathcal{M}(k)$:
\[0 \longrightarrow {\T}(M_2) \stackrel{f}{\longrightarrow} E \stackrel{g}{\longrightarrow} {\T}(M_1) \longrightarrow 0  \]
with $E=(E_\sigma,E_{\dR},E_{\ell}, 
I_{\sigma, {\dR}}, I_{{\overline \sigma}, {\ell}} ).$
In particular such an extension furnishes an extension in the Hodge realization, i.e. in the category $\mathcal{MHS}$ of mixed Hodge structures:
 \[0 \longrightarrow {\T}_\sigma(M_2) \stackrel{f_\sigma}{\longrightarrow} E_\sigma 
 \stackrel{g_\sigma}{\longrightarrow} {\T}_\sigma(M_1) \longrightarrow 0 . \]
According to Proposition~\ref{prop:hodge}, modulo isogenies there exists a unique extension $(M,i,j)$ of $M_1$ by $M_2$ which 
 defines the extension $E_\sigma$. In other worlds in the category $\mathcal{MHS}$ we have an isomorphism 
\[ \epsilon : E_\sigma \longrightarrow {\T}_{\sigma}(M) \] 
such that the following diagram commute
\begin{equation}
\xymatrix{
0 \ar[r]& {\T}_{\sigma}(M_2)  \ar[r]^{f_{\sigma}}\ar@{=}[d]&
E_{\sigma}  \ar[r]^{g_{\sigma}} \ar[d]_\epsilon&
 {\T}_{\sigma}(M_1)\ar[r] \ar@{=}[d]&0 \\
0 \ar[r]& {\T}_{\sigma}(M_2) \ar[r]^{{\T}_{\sigma}(i)}& {\T}_{\sigma}(M) \ar[r]^{{\T}_{\sigma}(j)}&{\T}_{\sigma}(M_1) \ar[r]&0
}
\label{eq:conj}
\end{equation}
where ${\T}_{\sigma}(i): {\T}_{\sigma}( M_2) \rightarrow {\T}_{\sigma}(M) $ and ${\T}_{\sigma}(j):{\T}_{\sigma}( M) \rightarrow {\T}_{\sigma}(M_1)$ are the morphisms in $\mathcal{MHS}$ induced by the morphisms of 1-motives $i: M_2 \rightarrow M$ and $j:M \rightarrow M_1$. The 1-motive $M$ underlying the extension $(M,i,j)$ is defined over $\CC$. Let  
$M_0$ be a model of $M$ over a finite extension $k'$ of $k$. Since by \cite{BLR90} 7.6 Proposition 5, the restriction of scalars ${\rm Res}_{k'/k}M_0$ is a 1-motive defined over $k$, we can assume that the 1-motive $M$ is in fact defined over $k$. 
By Propositions \ref{prop:l-adic} and \ref{prop:derham}, the extension $(M,i,j)$ of $M_1$ by $M_2$ defines extensions also in the $l$-adic and in the de Rham realizations. The Hodge, the de Rham and the $l$-adic realizations of the data $M,i: M_2 \rightarrow M$ and $j:M \rightarrow M_1$ build the following commutative diagrams with exact rows:
\[
\xymatrix{
0 \ar[r]& {\T}_{\ell}(M_2)\ar[r]^{{\T}_{\ell}(i)}& {\T}_{\ell}(M)\ar[r]^{{\T}_{\ell}(j)}&{\T}_{\ell}(M_1)\ar[r]&0 \\
0 \ar[r]& 
{\T}_{\sigma}(M_2)\otimes_{\QQ} {\QQ}_\ell  \quad \ar[r]^{{\T}_{\sigma}(i)\otimes {\QQ}_\ell}\ar[u]^{I_{{\overline \sigma}, {\ell}}}&
 \quad {\T}_{\sigma}(M) \otimes_{\QQ} {\QQ}_\ell \quad \ar[r]^{{\T}_{\sigma}(j)\otimes {\QQ}_\ell }\ar[u]^{I_{{\overline \sigma},{\ell}}}&
\quad {\T}_{\sigma}(M_1) \otimes_{\QQ}  {\QQ}_\ell\ar[r]\ar[u]^{I_{{\overline \sigma},{\ell}}}&0 
}
\]
\[
\xymatrix{
0 \ar[r]& {\T}_{\sigma}(M_2)\otimes_{\QQ} {\CC} \quad \ar[r]^{{\T}_{\sigma}(i)\otimes {\CC}}\ar[d]_{I_{\sigma, {\dR}}}&
\quad {\T}_{\sigma}(M)\otimes_{\QQ} {\CC} \quad \ar[r]^{{\T}_{\sigma}(j)\otimes {\CC}} \ar[d]_{I_{\sigma, {\dR}}}&
\quad {\T}_{\sigma}(M_1)\otimes_{\QQ} {\CC}\ar[r] \ar[d]_{I_{\sigma, {\dR}}}&0 \\
0 \ar[r]& {\T}_{\dR}(M_2)\otimes_k {\CC} \quad \ar[r]^{{\T}_{\dR}(i)\otimes {\CC}}& \quad {\T}_{\dR}(M)\otimes_k {\CC}  \quad \ar[r]^{{\T}_{\dR}(j)\otimes {\CC}}&
\quad {\T}_{\dR}(M_1) \otimes_k {\CC}\ar[r]&0
}
\]
We get therefore that the system of mixed realizations ${\T}(M)=({\T}_\sigma(M),{\T}_{\rm dR}(M),\\
{\T}_{\ell}(M),
I_{\sigma, {\rm dR}}, I_{{\overline \sigma}, {\ell}} )$ defined by $M$ is an extension of ${\T}(M_1)$ by ${\T}(M_2)$ in the category $\mathcal{M}(k)$. Because of the comparison isomorphisms and of the commutativity of diagram (\ref{eq:conj}), the isomorphism $\epsilon : E_\sigma \rightarrow {\T}_{\sigma}(M)$ implies the commutativity of the following diagram
for the $\ell$-adic realizations
\[
\xymatrix{
0 \ar[d] &  & & & 0 \ar[d] \\
{\T}_\ell(M_2)\ar@{=}[rrrr] \ar[dd]_{f_\ell}  &  & & & {\T}_{\ell}(M_2)\ar[dd]^{{\T}_{\ell}(i)}\\
&  & {\T}_{\sigma}(M_2)\otimes_{\QQ} {\QQ}_\ell \ar[dr]^{{\T}_{\sigma}(i) \otimes {\QQ}_\ell} \ar[dl]_{f_{\sigma}\otimes {\QQ}_\ell} \ar[ull]_{I_{{\overline \sigma}, {\ell}}} \ar[urr]^{I_{{\overline \sigma}, {\ell}}}& & \\ 
E_\ell \ar[dd]_{g_\ell}& E_{\sigma} \otimes_{\QQ} {\QQ}_\ell  \ar[l]_{I_{{\overline \sigma}, {\ell}}} \ar[rr]^{\epsilon \otimes {\QQ}_\ell}\ar[dr]_{g_{\sigma}\otimes {\QQ}_\ell} & &
{\T}_{\sigma}(M)\otimes_{\QQ} {\QQ}_\ell  \ar[r]^{I_{{\overline \sigma}, {\ell}}}\ar[dl]^{{\T}_{\sigma}(j) \otimes {\QQ}_\ell}  &
 {\T}_{\ell}(M) \ar[dd]^{{\T}_{\ell}(j)}\\
&  & {\T}_{\sigma}(M_1)\otimes_{\QQ} {\QQ}_\ell  \ar[dll]_{I_{{\overline \sigma}, {\ell}}} \ar[drr]^{I_{{\overline \sigma}, {\ell}}} & & \\
{\T}_\ell(M_1)\ar@{=}[rrrr] \ar[d] &  & & & {\T}_{\ell}(M_1) \ar[d]\\
0  &  & & & 0 
 }
\]
The reader can check that we have an analogous commutative diagram also for the de Rham realizations. The commutativity of these diagrams (together with the commutativity of diagram (\ref{eq:conj})) means that the system of realizations $E$ and ${\T}(M)$ are isomorphic as extensions of ${\T}(M_1)$ by $ {\T}(M_2)$. Therefore we have proved that any extension of ${\T}(M_1)$ by $ {\T}(M_2)$ in the category $\mathcal{M}(k)$ is defined by a unique 1-motive $M$ modulo isogenies. 
\end{proof}

\begin{rem}\label{essential} The hypothesis \emph{"coming from geometry"} in Deligne's conjecture is essential, because in the category ${\mathcal{MR}}(k)$ of mixed realizations there are too many extensions. In order to explain this fact, we construct an extension of ${\T}(\ZZ)$ by ${\T}(\GG_m)$ in the category  ${\mathcal{MR}}(k)$ which doesn't come from geometry. We start considering the 1-motive $M=[\ZZ \stackrel{u}{\rightarrow}  \GG_m], u(1)=2$, defined over $\QQ$, which is an extension of $\ZZ$ by $\GG_m.$ The mixed realization ${\T}(M)$ is the extension of ${\T}(\ZZ)$ by ${\T}(\GG_m)$ in the category of motives parametrized by the point $2$ of $\GG_m(\QQ)$, i.e. through the bijection 
\[ \GG_m(\QQ) \cong {\Ext}^1({\T}(\ZZ) ,{\T}(\GG_m) )\]
the extension ${\T}(M)$ corresponds to the point 2 of $\GG_m(\QQ)$.
Denote by $E=(E_\h,E_{\dR},\\
E_{\ell}, I_{\sigma, {\dR}}, I_{{\overline \sigma}, {\ell}} )$ the following mixed realization over $\textrm{Spec}(\QQ)$:
\begin{itemize}
	\item $E_{\dR}={\T}_{\dR}(M)$. In particular, $E_{\dR}=\QQ \oplus \QQ$ is the trivial extension of $\QQ$ by $\QQ$;
	\item $E_\h ={\T}_\h(M)$. In particular, the lattice $E_\ZZ$ underlying $E_\h$ is generated by $(\log(2),1), (2 \pi i,0) $ and it is a non trivial extension of $\ZZ$ by $\QQ(1)$
	\item $E_\ell =\ZZ_\ell(1) \oplus \ZZ_\ell$ is the trivial extension of $\ZZ_\ell$ by $\ZZ_\ell(1)$ for the Galois action $\mathrm{Gal}(\overline{\QQ} /\QQ);$ 
	\item $I_{\h, {\dR}}: E_\h \otimes_\QQ \CC \cong E_{\dR} \otimes_\QQ \CC$ is the comparison isomorphism underlying the mixed realization $ {\T}(M)$;
	\item  $I_{{\h}, {\ell}} : E_\h \otimes_\QQ \QQ_\ell \cong E_\ell $ is the comparison isomorphism defined sending $(\log(2),1)$ to $1 \in \ZZ_\ell$ and $(2 \pi i,0) $ to $\exp(\frac{2 \pi i}{l^n}) \in \ZZ_\ell(1)$.
\end{itemize}
This mixed realization $E$ is an extension of ${\T}(\ZZ)$ by ${\T}(\GG_m)$ in the category ${\mathcal{MR}}(\QQ)$ which isn't defined by a 1-motive extension of $\ZZ$ by $\GG_m$. 
\end{rem}



\begin{thebibliography}{99}

 
\bibitem[BK07]{BK07} L. Barbieri-Viale and B. Kahn, \emph{On the derived category of 1-motives I}, arXiv:0706.1498v1 [math.AG], 2007.
\bibitem[Bl87]{BL87} A. Beilinson, \emph{Height pairing between algebraic cycles},  Contemp. Math., vol. 67, 1987, pp. 1--24.
\bibitem[Be10]{Be10} C. Bertolin, \emph{Extensions of Picard stacks and their homological interpretation}, arXiv:1003.1866v1 [math.AG], 2010.
\bibitem[BLR90]{BLR90} S. Bosch, W. L\"utkebohmert, M. Raynaud, \emph{ N\'eron Models}, Ergebnisse der Mathematik und ihrer Grenzgebiete (3) [Results in Mathematics and Related Areas (3)], 21. Springer-Verlag, Berlin, 1990.
\bibitem[By83]{BY83} J.-L. Brylinski, \emph{1-motifs et formes automorphes
 (th\'eorie arithm\'etique des domaines 
des Siegel)}, Pub. Math. Univ. Paris VII 15, 1983.
\bibitem[D73]{SGA4} P. Deligne, \emph{La formule de dualit\'e globale}, Th\'eorie 		des topos et cohomologie \'etale des sch\'emas, Tome 3. S\'eminaire de G\'eom\'etrie Alg\'ebrique du Bois-Marie 1963--1964 (SGA 4). Lecture Notes in Mathematics, Vol. 305. Springer-Verlag, Berlin-New York, 1973, pp. 481-587.
\bibitem[D74]{D1} P. Deligne, \emph{Th\'eorie de Hodge III}, Inst. Hautes \'Etudes 		Sci. Publ. Math. No. 44, 1974, pp. 5--77.
\bibitem[D89]{D} P. Deligne, \emph{Le groupe fondamental de la droite
	projective moins trois points}, Galois groups over $\QQ$ (Berkeley, CA, 1987),  Math. Sci. Res. Inst. Publ., 16, Springer, New York, 1989, pp. 79--297.
\bibitem[G71]{G} J. Giraud, \emph{Cohomologie non ab\'elienne}, Die Grundlehren der 	mathematischen Wissenschaften, Band 179. Springer-Verlag, Berlin-New York, 1971.
\end{thebibliography}
\end{document}